\documentclass[proceedings]{aofa}

\usepackage{amsfonts}

\usepackage{amsthm}
\usepackage{graphicx}
\usepackage{epsfig}
\usepackage{mathrsfs}
\usepackage{amsmath}
\usepackage{bm}
\usepackage{epstopdf}
\usepackage{subcaption}
\usepackage{url}
\usepackage{tikz}

\newtheorem{theorem}{Theorem}
\newtheorem{lemma}{Lemma}
\newtheorem{example}{Example}

\renewcommand{\thefootnote}{\fnsymbol{footnote}}

\newcommand{\Po}[1]{\tilde{#1}}             
\newcommand\Me[1]{ {#1^*} }     

\def\R{ {\mathbb{R}}}

\newcommand\floor[1]{ \lfloor {#1} \rfloor}
\newcommand\dee[1]{ {\;\mathrm{d}#1}}

\def\lintersect{\cap}

\def\Z{ {\mathbb{Z}}}

\def\parsec{\par\noindent}

\def\med{\medskip\parsec}

\def\E{ {\mathbb{E}}}
\def\Var{ {\mathrm{Var}}}

\def\Normal{ {\mathcal{N}}}

\def\P{ {\mathcal{P}}}

\def\P{ {\frak{P}}}

\author[M. Drmota, A. Magner, and W. Szpankowski]{
	Michael Drmota\addressmark{1}\thanks{M. Drmota was supported by the Austrian Science Foundation FWF Grant No. S9604.} \and 
    Abram Magner\addressmark{2} \and 
    Wojciech Szpankowski\addressmark{3}\thanks{ A. Magner and W. 	
    Szpankowski were supported by
	NSF Center for Science of Information (CSoI) Grant CCF-0939370,
	and in addition by NSF Grants CCF-1524312, 
	NIH Grant 1U01CA198941-01, 
	and the NCN grant, grant  UMO-2013/09/B/ST6/02258.}\thanks{ 
    W. Szpankowski is also with the Faculty of Electronics, Telecommunications, and Informatics, Gda\'nsk University of Technology, Poland.}
}

\title[Asymmetric R\'enyi Problem]{Asymmetric R\'enyi Problem and PATRICIA Tries}

\address{\addressmark{1} Institute for Discrete Mathematics and Geometry, TU Wien, A-1040 Wien, Austria (email: michael.drmota@tuwien.ac.at). \\
	\addressmark{2} Coordinated Science Lab, UIUC, Champaign, IL, USA (email: anmagner@illinois.edu) \\
    \addressmark{3} Department of Computer Science, Purdue University, IN 47907, USA (email: spa@cs.purdue.edu)
}


\begin{document}
\maketitle


\begin{abstract}
\paragraph{Abstract: }
In 1960 R\'enyi asked for the number 
of random queries necessary to recover a hidden bijective labeling of  
$n$ distinct objects.  In each query one selects a random subset of labels and asks, 
what is the set of objects that have these labels?
We consider here an asymmetric version of
the problem in which in every query an object is chosen with probability 
$p > 1/2$ and we ignore ``inconclusive'' queries.
We study the number of queries needed to recover the labeling in 
its entirety (the \emph{height}), to recover at least one single element 
(the \emph{fillup level}), and to recover a randomly chosen element 
(the \emph{typical depth}).  
This problem exhibits several remarkable behaviors: 
the depth $D_n$ converges in probability but not almost surely and while it
satisfies the central limit theorem its local limit theorem doesn't hold;
the height $H_n$ and the fillup level $F_n$
exhibit phase transitions with respect
to $p$ in the second term.   
To obtain these results, we take a unified approach via the analysis of the 
\emph{external profile} defined at level $k$ as the number of 
elements recovered by the $k$th query. 
We first establish new precise asymptotic results for
the average and variance, and a central limit law, for the external profile in the regime
where it grows polynomially with $n$.
We then extend the external profile results to the boundaries of the central region, 
leading to the solution of our problem for the height and fillup level.
As a bonus, our analysis implies novel results for random PATRICIA tries, as 
it turns out that the problem is probabilistically equivalent to
the analysis of the height, fillup level, typical depth, and external profile of a PATRICIA trie built from $n$ independent
binary sequences generated by a biased($p$) memoryless source.
\end{abstract}

\renewcommand{\thefootnote}{\arabic{footnote}}
\setcounter{footnote}{0}

\section{Introduction}
In his lectures in the summer of 1960 at Michigan State University,
Alfred R\'enyi discussed several problems related to random sets \cite{renyi}.
Among them there was a problem regarding recovering a labeling of a set $X$ of $n$
distinct objects by asking random subset questions of the form ``which objects correspond to the labels in the (random) set $B$?''
For a given method of randomly selecting queries, R\'enyi's original problem asks for the typical behavior of the number of queries
necessary to recover the hidden labeling.

Formally, the unknown labeling of the set $X$ is a bijection $\phi$ from $X$ to a set $A$ of labels 
(necessarily with equal cardinality $n$), and a query takes the form of a subset $B \subseteq A$.  
The response to a query $B$ is $\phi^{-1}(B) \subseteq X$.  

Our contribution in this paper is a precise analysis of several parameters of R\'enyi's problem for a particular natural probabilistic model
on the query sequence.  In order to formulate this model precisely, it is convenient to first state a view of the process
that elucidates its tree-like structure.  In particular, 
a sequence of queries corresponds to a refinement of partitions of the set of objects, where two objects are in
different partition elements if they have been distinguished by some sequence of queries.
More precisely, the refinement works as follows: 
before any questions are asked, we have a trivial partition $\P_0 = X$ consisting of a single class (all objects).  Inductively,
if $\P_{j-1}$ corresponds to the partition induced by the first $j-1$ queries, then $\P_j$ is constructed from $\P_{j-1}$ by splitting
each element of $\P_{j-1}$ into at most two disjoint subsets: those objects that are contained in the preimage of the $j$th query set $B_j$ and
those that are not.  The hidden labeling is recovered precisely when 
the partition of $X$ consists only of singleton elements.  An instance of 
this process may be viewed as a rooted binary tree (which we call the 
\emph{partition refinement tree}) in which the $j$th level, for $j \geq 0$, corresponds
to the partition resulting from $j$ queries; a node in a level corresponds to an element of that 
partition.  A right child corresponds to a subset of a parent partition element that is included
in the subsequent query, and a left child corresponds to a subset that is not included.  
See Example~\ref{QuerySequenceExample} for an illustration. 
\begin{example}[Demonstration of partition refinement]
	\label{QuerySequenceExample}
    Consider an instance of the problem where $X = [5] = \{1, ..., 5\}$, 
    with labels $(d, e, a, c, b)$ respectively (so $A = \{a, b, c, d, e\}$).
    Consider the following sequence of queries:
	    \tikzstyle{level 1}=[level distance=1.0cm, sibling distance=2.5cm]
	    \tikzstyle{level 2}=[level distance=1.0cm, sibling distance=1.5cm]
	    \tikzstyle{bag} = [rectangle, minimum width=3pt,inner sep=0pt]
	    \tikzstyle{end} = [circle, minimum width=3pt,fill, inner sep=0pt]
		
    \begin{minipage}{0.5\textwidth}
    \begin{enumerate}
    	\item
        	$B_1 = \{b, d\} \mapsto \{1, 5\}$
        \item
        	$B_2 = \{a, b, d\} \mapsto \{1, 3, 5\}$,
        \item
        	$B_3 = \{a, c, d\} \mapsto \{1, 3, 4\}$,
    \end{enumerate}
    \end{minipage}
    \vspace{20pt}
    \begin{minipage}{0.5\textwidth}
	    \begin{tikzpicture}[grow=down]
	    \node[bag]{\{1, 2, 3, 4, 5\}}
	    	child{
	            node[bag]{\{2,3,4\}}
	        	child{ 
	            	node[bag]{\{2,4\}}
                    child{ 
                        node[rectangle,draw]{2}
                    }
                    child{ 
                        node[rectangle,draw]{4}
                    }
	            }
	            child{ 
	            	node[rectangle,draw]{3}
	            }
	        }
	        child{
	            node[bag]{\{1,5\}}
                child{
	                node[bag,right]{\{1, 5\}}
			        	child{ 
			            	node[rectangle,draw]{5}
			                edge from parent
			                node[above]{}
			            }
			            child{ 
			            	node[rectangle,draw]{1}
			                edge from parent
			                node[above]{}
			            }
                }
	        };
	\end{tikzpicture}
	\label{QuerySequenceExampleDiagram}
    \end{minipage}
    Each level $j\geq 0$ of the tree depicts the partition $\P_j$, where a right child node 
    corresponds to the subset of objects in the parent set which are contained in the 
    response to the $j$th query.  Singletons are only explicitly depicted in the first
    level in which they appear.
    \qed
\end{example}
In this work we consider a version of
the problem in which, in every query, each label is included independently
with probability $p > 1/2$ (the \emph{asymmetric case}) and we \emph{ignore inconclusive queries}.  In particular, if a candidate query fails to nontrivially split some element 
of the previous partition, we modify the query by deciding again independently 
whether or not to include each label of that partition element with probability $p$.  
We perform this modification until
the resulting query splits every element of the previous partition nontrivially.  
See Example~\ref{QuerySequenceIgnoreExample}.
\begin{example}[Ignoring inconclusive queries]
	\label{QuerySequenceIgnoreExample}
    Continuing Example~\ref{QuerySequenceExample}, the query $B_2$ fails to split the 
    partition element $\{1, 5\}$, so it is an example of an inconclusive query and would 
    be modified in our model to, say, $B_2' = \phi(\{1,3\})$.
    The resulting refinement of partitions is depicted as a tree here.  Note that the tree
    now does not contain non-branching paths and that $B_2$ is ignored in the final query
    sequence.
    \begin{minipage}{0.5\textwidth}
    \begin{enumerate}
    	\item
        	$B_1 = \{b, d\} \mapsto  \{1,5\}$
        \item
        	$B_2' = \{a,d\} \mapsto \{1,3\}$
        \item
        	$B_3 = \{a, c, d\} \mapsto \{1, 3, 4\}$.
    \end{enumerate}
    \end{minipage}
    \begin{minipage}{0.5\textwidth}
    	\vspace{12pt}
 	    \tikzstyle{level 1}=[level distance=1.0cm, sibling distance=2.5cm]
	    \tikzstyle{level 2}=[level distance=1.0cm, sibling distance=1.5cm]
	    \tikzstyle{bag} = [rectangle, minimum width=3pt,inner sep=0pt]
	    \tikzstyle{end} = [circle, minimum width=3pt,fill, inner sep=0pt]
 	    \begin{tikzpicture}[grow=down]
	    \node[bag]{\{1, 2, 3, 4, 5\}}
	    	child{
	            node[bag]{\{2,3,4\}}
	        	child{ 
	            	node[bag]{\{2,4\}}
                    child{ 
                    	node[rectangle,draw]{2}
                    }
                    child{ 
                    	node[rectangle,draw]{4}
                    }
	            }
	            child{ 
	            	node[rectangle,draw]{3}
	            }
	        }
	        child{
	            node[bag]{\{1,5\}}
	        	child{ 
	            	node[rectangle,draw]{5}
	                edge from parent
	                node[above]{}
	            }
	            child{ 
	            	node[rectangle,draw]{1}
	                edge from parent
	                node[above]{}
	            }
	        };
	\end{tikzpicture}   
    \end{minipage}
    \qed
\end{example}
We study three parameters of this random process: $H_n$, the number of such queries
needed to recover the entire labeling; $F_n$, the number needed before at
least one element is recovered; 
and $D_n$, the number needed to recover an element selected uniformly
at random.  Our objective is to present precise probabilistic estimates 
of these parameters and to study the distributional behavior of $D_n$.

The symmetric version (i.e., $p=1/2$) of the problem (with a variation) was discussed by Pittel 
and Rubin in \cite{pittelrubin1990}, where they analyzed the typical value of $H_n$.  
In their model, a query is constructed by deciding
whether or not to include each label from $A$ independently with probability $p=1/2$.
To make the problem interesting, they added a constraint similar to ours: namely, a
query is, as in our model, admissible if and only if it splits every nontrivial element of the current
partition.  In contrast with our model, however, Pittel and Rubin completely discard inconclusive queries (rather than modifying their inconclusive subsets as we do).  
Despite this difference, the model considered in \cite{pittelrubin1990} is
probabilistically equivalent to ours for the symmetric case.  Our primary contribution
is the analysis of the problem in the asymmetric case ($p > 1/2$), but our methods
of proof allow us to recover the results of Pittel and Rubin. 

The question asked by R\'enyi brings some surprises.  For the
symmetric model ($p=1/2$) Pittel and Rubin \cite{pittelrubin1990} were
able to prove that the number of necessary queries is with high probability
(whp) (see Theorem~\ref{HeightFillupTheorem}) 
\begin{align}
	\label{e1}
	H_n = \log_{2} n +\sqrt{2\log_{2} n} +o(\sqrt{\log n}).
\end{align}
In this paper, we re-establish this result using a different approach
\emph{and} prove that for $p > 1/2$ the number
of queries grows whp as
\begin{align}
	\label{e2}
    H_n = \log_{1/p} n +\frac{1}{2}\log_{p/q} \log n  +o(\log \log  n),
\end{align}

where $q:=1-p$.
Note a phase transition in the second term.
We show that a similar phase transition occurs in the asymptotics for $F_n$
(see Theorem~\ref{HeightFillupTheorem}):
\begin{align} 
\label{e3}
    F_n =
    \begin{cases}
        \log_{1/q} n - \log_{1/q}\log\log n + o(\log\log\log n)     &      p > q \\
        \log_{2} n - \log_2\log n + o(\log\log n)                   &      p = q = 1/2.
    \end{cases}    
\end{align}
We then prove in Theorem~\ref{DepthTheorem} some interesting probabilistic behaviors
of $D_n$.
We have $D_n/\log n  \to 1/h(p)$ (in probability)
where $h(p) := -p\log p - q\log q$, but we do not have almost sure
convergence.  Moreover, $D_n$ appropriately normalized
satisfies a central limit result, but 
not a local limit theorem due to some oscillations discussed below.

We establish these results in a novel way by considering first the \emph{external profile}
$B_{n,k}$, whose analysis was, until recently, an open problem of its own (the second 
and third authors gave a precise analysis of the external profile in an important range 
of parameters in \cite{magnerPhD2015,ms2015}, but the present paper requires nontrivial 
extensions).  
The external profile at level $k$ is
the number of bijection elements revealed by the $k$th query (one may also define the \emph{internal} profile
at level $k$ as the number of non-singleton elements of the partition immediately after the $k$th query).  Its study is motivated by
the fact that many other parameters, including all of those that we mention here, can
be written in terms of it. Indeed, $\Pr[D_n=k]=\E[B_{n,k}]/n$, 
$H_n = \max\{k : ~ B_{n,k} > 0\}$, and $F_n = \min\{k : ~ B_{n,k} > 0\} - 1$.

We now discuss our new results concerning the probabilistic  
behavior of the external profile.  
We establish in \cite{ms2015,magnerPhD2015} precise asymptotic expressions for the 	
expected value and variance of $B_{n,k}$ in the \emph{central range}, that is, 
with $k \sim \alpha\log n$, where, for any fixed $\epsilon > 0$,
$\alpha \in (1/\log(1/q) + \epsilon, 1/\log(1/p) - \epsilon)$
(the left and right endpoints of this interval are associated
with $F_n$ and $H_n$, respectively).  Specifically,
we show that both the mean and the variance
are of the same (explicit) polynomial order of growth
(with respect to $n$) 
(see Theorem~\ref{CentralRangeTheorem}).
More precisely, we show that both expected value and variance grow
for $k\sim\alpha \log n$ as
$$
H(\rho(\alpha), \log_{p/q}(p^kn)) ~  \frac{n^{\beta(\alpha)}}{\sqrt{C \log n}}
$$
where $\beta(\alpha)\leq 1$ and $\rho(\alpha)$ are 
complicated functions of $\alpha$, 
$C$ is an explicit constant, and $H(\rho, x)$ is
a function that is periodic in $x$. 
The oscillations come from infinitely many
regularly spaced saddle points that we observe when inverting the Mellin
transform of the Poisson generating function of $\E[B_{n,k}]$.  
Finally, we prove a central limit theorem;
that is,
${(B_{n,k} - \E[B_{n,k}])}/{\sqrt{\Var[B_{n,k}]}} \to \Normal(0,1) $
where $\Normal(0,1)$ represents the standard normal distribution.  

In the present paper,
we exploit the expected value analysis of $B_{n,k}$ in the central range to give precise
distributional information about $D_n$ via the identity
$\Pr[D_n = k] = \E[B_{n,k}]/n$.  Note that the oscillations in $\E[B_{n,k}]$
are the source of the peculiar behavior of $D_n$.

In order to establish the most interesting results claimed in the present paper
for $H_n$ and $F_n$, the analysis sketched above does not suffice: we need to estimate the
mean and the variance of the external profile \emph{beyond} the range
$\alpha \in (1/\log(1/q) + \epsilon, 1/\log(1/p) - \epsilon)$; in particular,
for $F_n$ and $H_n$ we need expansions at the left and right side, respectively, of this range.
This, it turns out, requires a novel approach and analysis, as discussed in detail in our forthcoming
journal paper \cite{drmotamagnerszpa2016},		
leading to the announced results on the R\'enyi problem in (\ref{e2}) and (\ref{e3}).  

Having described most of our main results, we mention an important equivalence
pointed out by Pittel and Rubin \cite{pittelrubin1990}.
They observed that their version of the R\'enyi process
resembles the construction of a digital tree
known as a PATRICIA trie\footnote{We recall that a PATRICIA trie is a trie in
which non-branching paths are \emph{compressed}; that is, there are no unary paths.}
\cite{knuth1998acp,szpa2001Book}. In fact,
the authors of \cite{pittelrubin1990} show that $H_n$ is probabilistically
equivalent to the height (longest path) of a PATRICIA trie built from
$n$ binary sequences generated independently by a memoryless source with bias $p=1/2$
(that is, with a ``1'' generated with probability $p$; this is often called
the \emph{Bernoulli model with bias $p$}); the equivalence is true more generally, for $p \geq 1/2$.
It is easy to see that
$F_n$ is equivalent to the fillup level (depth of the deepest full level),
$D_n$ to the typical depth (depth of a randomly chosen leaf), and $B_{n,k}$
to the external profile of the tree (the number of leaves at level $k$; the internal
profile at level $k$ is similarly defined as the number of non-leaf nodes at that 
level).  We spell out this equivalence in the following simple claim.  
\begin{lemma}[Equivalence of parameters of the R\'enyi problem with those of PATRICIA tries]
	\label{EquivalenceLemma}
    Any parameter (in particular, $H_n, F_n, D_n$, and $B_{n,k}$) of the R\'enyi process with bias $p$ 
    that is a function of the partition refinement
    tree is equal in distribution to the same function of a random PATRICIA trie generated
    by $n$ independent infinite binary strings from a memoryless source with bias $p \geq 1/2$.
\end{lemma}    
\begin{proof}
    In a nutshell, we
    couple a random PATRICIA trie and the sequence of queries from the R\'enyi process
    by constructing both from the same sequence of binary strings from a memoryless source.
    We do this in such a way that the resulting PATRICIA trie and the partition refinement
    tree are isomorphic with probability $1$, so that parameters defined in terms of
    either tree structure are equal in distribution. 
    
    More precisely, we start with $n$ independent infinite binary strings $S_1, ..., S_n$ 
    generated according to a memoryless source with bias $p$, where each string corresponds to a
    unique element of the set of labels (for simplicity, we assume that $A = [n]$, and
    $S_j$ corresponds to $j$, for $j \in [n]$).  These induce a PATRICIA trie $T$,
    and our goal is to show that we can simulate a R\'enyi process using these strings,
    such that the corresponding tree $T_R$ is isomorphic to $T$ as a rooted plane--
    oriented tree (see Example~\ref{QuerySequenceIgnoreExample}).  The basic idea is as follows: we maintain for each string $S_j$ an
    index $k_j$, initially set to $1$.  Whenever the R\'enyi process demands
    that we make a decision about whether or not to include label $j$ in a query,
    we include it if and only if $S_{j,k_j} = 1$, and then increment $k_j$ by $1$.
    
    Clearly, this scheme induces the correct distribution on queries.  Furthermore,
    the resulting partition refinement tree (ignoring inconclusive queries) is easily 
    seen to be isomorphic to $T$.
    Since the trees are isomorphic, the parameters of interest are equal in each case.
\end{proof}
Thus, our results on these parameters for the R\'enyi problem directly lead
to novel results on PATRICIA tries, and vice versa.
In addition to their use as data structures, PATRICIA tries also arise as 
combinatorial structures which capture the behavior of various processes 
of interest in computer science and information theory
(e.g., in leader election processes without trivial splits 
\cite{jansonszpa1996} and in the solution to R\'{e}nyi's problem which we study here
\cite{pittelrubin1990, devroye1992}).

Similarly, the version of the R\'enyi problem that allows inconclusive queries corresponds
to results on tries built on $n$ binary strings from a memoryless source.  We thus discuss
them in the literature survey below.


Now we briefly review known facts about PATRICIA tries and other digital trees when built over
$n$ independent strings generated by a memoryless source.
Profiles of tries in both the asymmetric and symmetric cases were studied 
extensively in \cite{park2008}.  The expected profiles
of digital search trees in both cases were analyzed in \cite{drmotaszpa2011}, and
the variance for the asymmetric case was treated
in \cite{kazemi2011}.  Some aspects of trie and PATRICIA trie profiles 
(in particular, the concentration of their distributions) were
studied using probabilistic methods in \cite{devroye2004, devroye2002}.
The depth in PATRICIA for the symmetric model was analyzed in 
\cite{devroye1992,knuth1998acp} while for the asymmetric model in 
\cite{szpa1990}. The leading asymptotics for the PATRICIA height
for the symmetric Bernoulli model was first analyzed by Pittel   
\cite{Pittel85} (see also \cite{szpa2001Book} for suffix trees).
The two-term expression for the height of PATRICIA for the symmetric model 
was first presented in \cite{pittelrubin1990} as discussed above
(see also \cite{devroye1992}).  Finally, in \cite{magnerPhD2015,ms2015}, the second two 
authors of the present paper presented a precise analysis of the external profile (including
its mean, variance, and limiting distribution) in the asymmetric case, for the range
in which the profile grows polynomially.  The present work relies on this previous analysis,
but the analyses for $H_n$ and $F_n$ involve a significant extension, since they rely on 
precise asymptotics for the external profile outside this central range.

Regarding methodology, the basic framework (which we use here) for analysis of digital tree recurrences by 
applying the Poisson transform to derive a functional equation, converting this to an algebraic
equation using the Mellin transform, and then inverting using the saddle point method/singularity
analysis followed by depoissonization, was worked out in \cite{drmotaszpa2011} and followed in \cite{park2008}.
While this basic chain is common, the challenges of applying it vary dramatically between the different
digital trees, and this is the case here.  As we discuss later (see (\ref{pg1}) and the surrounding text), this variation starts with the quite different
forms of the Poisson functional equations, which lead to unique analytic challenges.  

The plan for the paper is as follows. In the next section we formulate
more precisely our problem and present our main results regarding the
external profile, height, fillup level, and depth. Sketches of proofs are
provided in the last section (the full proofs are provided in the journal
version of this paper).  
 
\section{Main Results}
\label{MainResults}

In this section, we formulate precisely R\'enyi's problem and 
present our main results.  
Our goal is to provide precise asymptotics for three natural parameters 
of the R\'enyi problem on $n$ objects with each label in a given 
query being included with probability $p \geq 1/2$:
the number $F_n$ of queries needed to identify at least one single element of 
the bijection, the number $H_n$ needed to recover the bijection 
in its entirety, and the number $D_n$ needed to recover an 
element of the bijection chosen uniformly at random from the $n$ objects.  
If one wishes to determine the label for a particular object, 
these quantities correspond to the best, worst, and average case performance, 
respectively, of the random subset strategy proposed by R\'enyi.  
We call these parameters, the fillup level $F_n$, the height $H_n$, and
the depth $D_n$, respectively (these names come from the corresponding quantities
in random digital trees).
One more parameter is relevant: we can present a unified analysis of 
our main three parameters $F_n, H_n,$ and $D_n$ via the \emph{external profile} 
$B_{n,k}$, which is the number of elements of the bijection on $n$ items 
identified by the $k$th query.

Our analysis reveals several remarkable behaviors:  
the depth $D_n$ converges in probability but not almost surely and while it
satisfies the central limit theorem its local limit theorem doesn't hold.
Perhaps most interestingly, the height $H_n$ and the fillup level $F_n$
exhibit phase transitions with respect 
to $p$ in the second term.  

To begin, we recall the relations of $F_n$, $H_n$, and $D_n$ to $B_{n,k}$:
\begin{align*}
    F_n = \min\{k : ~ B_{n,k} > 0\} - 1 &&
    H_n = \max\{k : ~ B_{n,k} > 0\} &&
    \Pr[D_n = k] = \E[B_{n,k}]/n.
\end{align*}
Using the first and second moment methods, we can then obtain upper and lower bounds
on $H_n$ and $F_n$ in terms of the moments of $B_{n,k}$:
\begin{eqnarray*}
    \Pr[H_n>k] \leq \sum_{j>k} \E[B_{n,j}],  \ \ \  \ \ 
    \Pr[H_n<k] \leq \frac{\Var[B_{n,k}]}{\E[B_{n,k}]^2},
\end{eqnarray*}
and
\begin{eqnarray*}
    \Pr[F_n > k] \leq \frac{\Var[B_{n,k}]}{\E[B_{n,k}]^2}, &&
    \Pr[F_n < k] \leq \E[B_{n,k}].
\end{eqnarray*}
The analysis of the distribution of $D_n$ reduces simply to that of $\E[B_{n,k}]$.

In the next section, we show that the fillup level $F_n$ and 
the height $H_n$ have the following precise asymptotic expansions.
Both exhibit a phase transition with respect to $p$ in the second term.
A complete proof can be found in our journal version of this paper \cite{drmotamagnerszpa2016}.

\begin{theorem}[Asymptotics for $F_n$ and $H_n$]
    \label{HeightFillupTheorem}
    With high probability,
    \begin{align}
        H_n =
        \begin{cases}
                \log_{1/p} n + \frac{1}{2}\log_{p/q}\log n + o(\log\log n) &
                p > q \\
                \log_{2} n + \sqrt{2\log_2 n} + o(\sqrt{\log n}) &
                p = q
        \end{cases}    
    \end{align}
    and
    \begin{align}
        F_n =
        \begin{cases}
            \log_{1/q} n - \log_{1/q}\log\log n + o(\log\log\log n)  &
            p > q \\
            \log_{2} n - \log_2\log n + o(\log\log n) &
            p = q 
        \end{cases}    
    \end{align}
for large $n$.
\end{theorem}
While the behavior of the fillup level $F_n$ could be anticipated 
\cite{pittel86}
(by comparing it to the corresponding result in the version of R\'enyi's problem 
allowing inconclusive queries), the behavior of the height $H_n$ is rather
more unusual.  It is difficult to compare the height result to the analogous quantity for
tries or digital search trees, because only the first term is given for 
$p > 1/2$ in the literature: for tries, it is $\frac{2}{\log(1/(p^2 + q^2))} \log n$,
while for digital search trees it is $\log_{1/p} n$, as in PATRICIA tries.

Focusing on the second term of each expression given in the theorem, this result
says that the deviation of the typical height from $\log_{1/p} n$ is asymptotically
larger when $p=1/2$ than when $p > 1/2$.  That is, the height of the tallest fringe
subtree (i.e., a subtree rooted near $\log_{1/p} n$) is asymptotically larger in
the symmetric case.  A complete explanation of this phenomenon would likely require
consideration of the number of such subtrees (i.e., the internal profile at level
$\log_{1/p} n$) and the number of strings participating
in each of them.  In the language of the R\'enyi problem, this latter parameter is
the number of objects that remain unidentified after approximately $\log_{1/p} n$ queries.

Moving to the number of questions $D_n$ needed to identify a random 
element of the bijection, we have the following theorem (note that 
due to the evolution process
of the random PATRICIA trie, all random variables can be defined on the same
probability space).  
\begin{theorem}[Asymptotics and distributional behavior of $D_n$]
    \label{DepthTheorem}
    For $p > 1/2$, the normalized depth $D_n/\log n$ converges in probability to $1/h(p)$,
    where $h(p) := -p\log p - q\log q$ is the Bernoulli entropy function, but not
    almost surely. In fact,  
    $$
        \liminf_{n\to\infty} D_n/\log n = 1/\log(1/q) \ \ \ \ \ \mbox{(a.s)} \ \ \ \ 
        \limsup_{n\to\infty} D_n/\log n = 1/\log(1/p).
    $$
    Furthermore, $D_n$ satisfies a central limit theorem; that is, 
    $(D_n - \E[D_n])/\sqrt{\Var[D_n]} \to \Normal(0, 1)$,
    where $\E[D_n] \sim \frac{1}{h(p)}\log n$ and $\Var[D_n] \sim c\log n$ where $c$
    is an explicit constant. 
    A \emph{local} limit theorem does not hold:
    for $x=O(1)$ and $k = \frac{1}{h} (\log n + x\sqrt{\kappa_*(-1)\log n/h})$,
    where $\kappa_*(-1)$ is some explicit constant and $h = h(p)$, we obtain
    $$
        \Pr\left[D_n = k\right]  
        \sim  H\left(-1;\log_{p/q} p^kn \right)
        \frac{e^{-x^2/2}}{\sqrt{2\pi C \log n}} 
    $$
    for an oscillating function
    $H(-1;\log_{p/q} p^kn )$ (see Figure~\ref{HFigure}) defined in 
    Theorem~\ref{CentralRangeTheorem} below and an explicitly known constant $C$.
\end{theorem}
\begin{figure}[h]
    \begin{center}
        \begin{subfigure}[b]{0.375\textwidth}
\includegraphics[width=\textwidth]{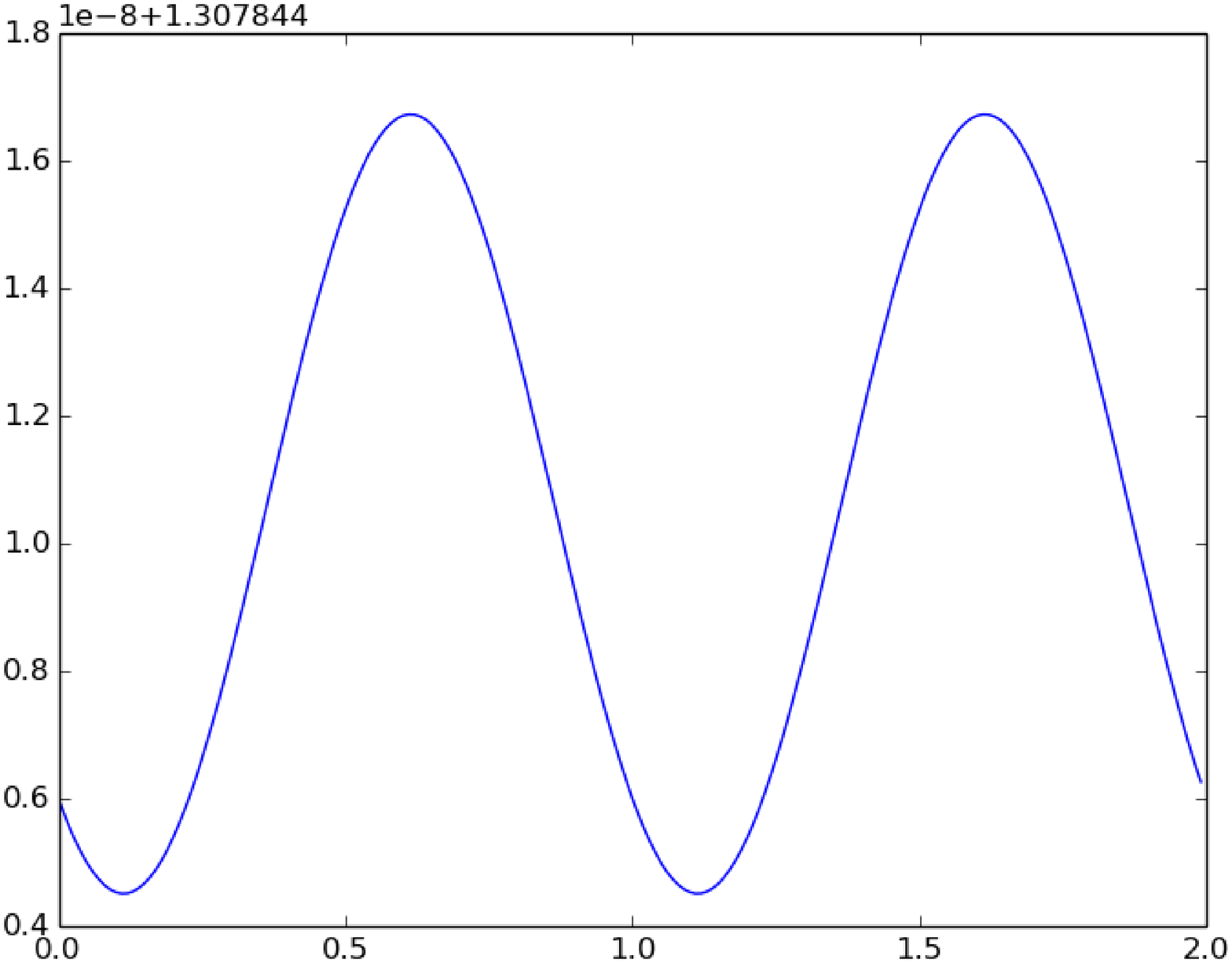}
        \end{subfigure}%
        \begin{subfigure}[b]{0.375\textwidth}
            \includegraphics[width=\textwidth]{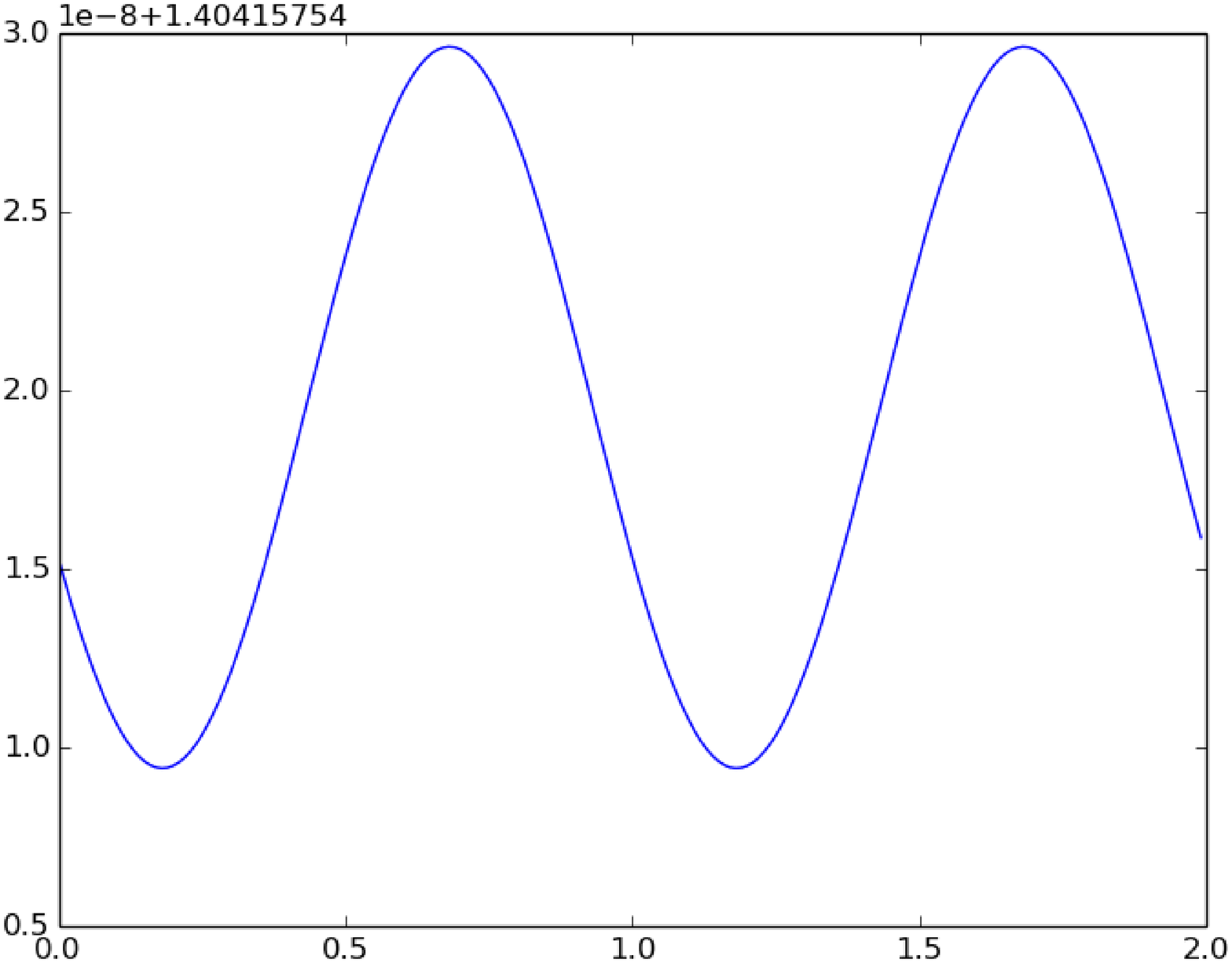}
        \end{subfigure}%
        \begin{subfigure}[b]{0.375\textwidth}
            \includegraphics[width=\textwidth]{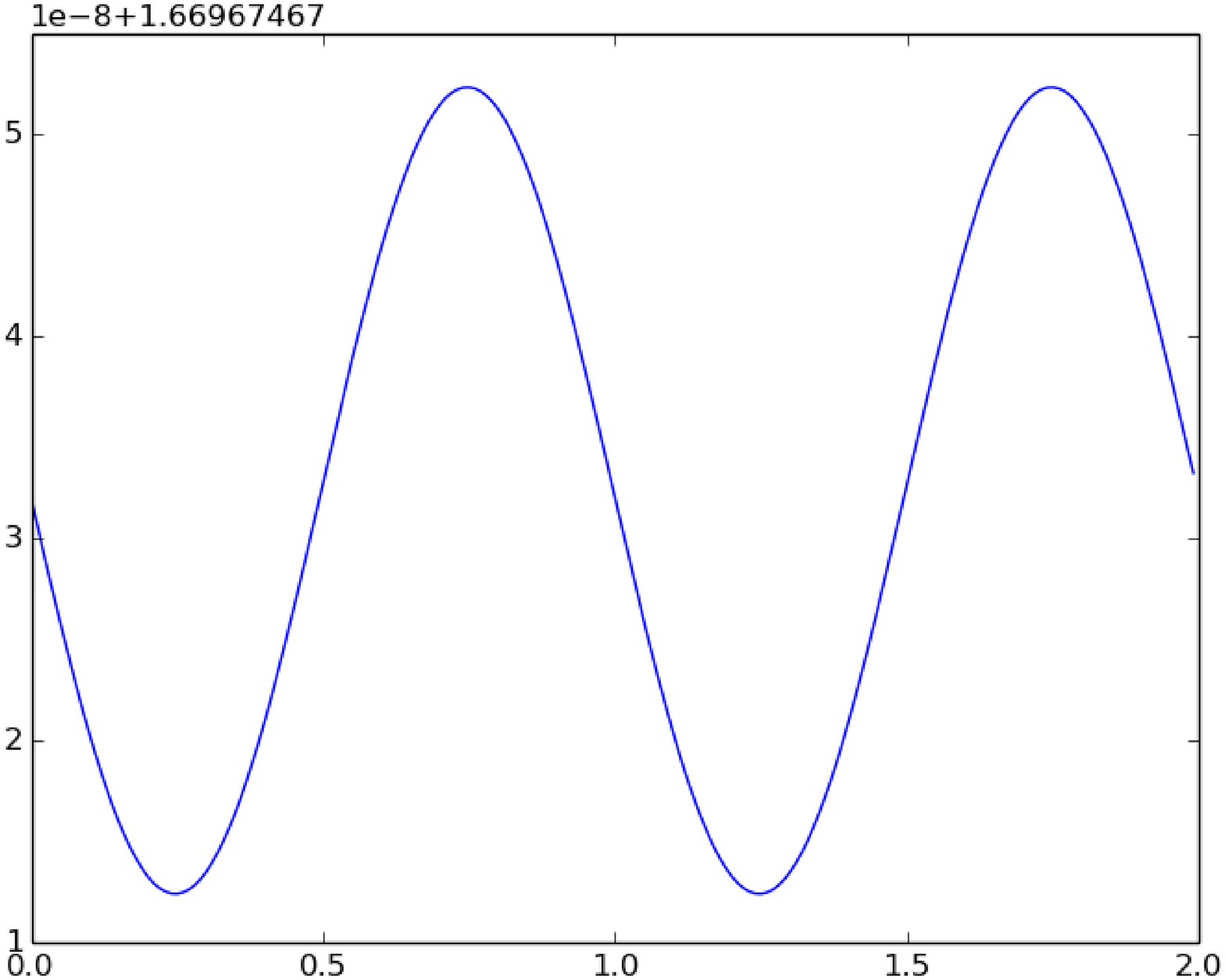}
        \end{subfigure}%
    \end{center}
    \caption{
        Plots of $H(\rho, x)$ for $\rho = -0.5, 0, 0.5$.
    }
    \label{HFigure}
\end{figure}

Again, the depth exhibits a phase transition: for $p=1/2$ we have
$D_n/\log n \to 1/\log 2$ almost surely, which doesn't hold for $p > 1/2$.  We note
that some of the results on the depth (namely, the convergence in probability and
the central limit theorem) are already known (see \cite{rais1993}), but our contribution is  
a novel derivation of these facts via the profile analysis.  Qualitatively, the oscillatory
behavior of the external profile that is responsible for the lack of local limit theorem for
the depth occurs also in both tries and digital search trees.

We now explain our approach to the analysis of the moments of $B_{n,k}$ in appropriate
ranges (we follow \cite{magnerPhD2015,ms2015}).  For this, we take an analytic approach 
\cite{Flajoletsedgewick2009,szpa2001Book}.
We first explain it for the analysis relevant
to $D_n$, and then show how to extend it for $H_n$ and $F_n$.
More details can be found in the next section.

We start by deriving a recurrence for the average profile, 
which we denote by $\mu_{n,k} := \E[B_{n,k}]$. It satisfies
\begin{align}
    \label{muRecurrence}
    \mu_{n,k} = (p^n + q^n)\mu_{n,k} + \sum_{j=1}^{n-1} { n\choose j } p^j q^{n-j} (\mu_{j,k-1} + \mu_{n-j,k-1})
\end{align}
for $n\geq 2$ and $k \geq 1$, with some initial/boundary conditions; 
most importantly, $\mu_{n,k} = 0$ for $k \geq n$ and any $n$.
Moreover, $\mu_{n,k} \leq n$ for all $n$ and $k$ owing to the elimination of
inconclusive queries.  This recurrence 
arises from conditioning on the number $j$ of objects that are included 
in the first query.  If $1 \leq j \leq n-1$ objects are included,
then the conditional expectation is a sum of contributions from those 
objects that are included and those that aren't.  If, on
the other hand, all objects are included or all are excluded from the 
first potential query (which happens with probability $p^n + q^n$),
then the partition element splitting constraint on the queries applies, 
the potential query is ignored as inconclusive, and the contribution is $\mu_{n,k}$.

The tools that we use to solve this recurrence (for details see \cite{magnerPhD2015,ms2015}) 
are similar to those of the analyses for digital trees \cite{szpa2001Book} 
such as tries and digital search trees (though the analytical details 
differ significantly). We first derive a functional equation for the 
Poisson transform 
$\Po{G}_k(z) = \sum_{m \geq 0} \mu_{m,k}\frac{z^m}{m!}e^{-z}$ of
$\mu_{n,k}$, which gives 
\[
    \Po{G}_k(z) = \Po{G}_{k-1}(pz) + \Po{G}_{k-1}(qz) + e^{-pz}(\Po{G}_k - 
\Po{G}_{k-1})(qz) + e^{-qz}(\Po{G}_{k} - \Po{G}_{k-1})(pz).
\]
This we write as
\begin{align}
\label{pg1}
    \Po{G}_k(z) = \Po{G}_{k-1}(pz) + \Po{G}_{k-1}(qz) + \Po{W}_{k,G}(z), 
\end{align}
We contrast this functional equation with those for tries \cite{park2008} and for 
digital search trees \cite{drmotaszpa2011}:
in tries, the expression $\Po{W}_{k,G}(z)$ does not appear, which significantly
simplifies the analysis in that case.  In digital search trees, the functional equation
is a differential equation, and the analysis is consequently quite different.

At this point the goal is to determine asymptotics for 
$\Po{G}_k(z)$ as $z\to\infty$ in a cone around the positive real axis.
When solving (\ref{pg1}), $\Po{W}_{k,G}(z)$ complicates the analysis because it has
no closed-form Mellin transform (see below); we handle it via its Taylor series.
Finally, depoissonization \cite{szpa2001Book} will allow us to transfer the
asymptotic expansion for $\Po{G}_k(z)$ back to one for $\mu_{n,k}$:
\begin{align*}
	\mu_{n,k}
    = \Po{G}_k(n) - \frac{n}{2}\Po{G}''_k(n) + O(n^{\epsilon-1}).
\end{align*}

To convert (\ref{pg1}) to an algebraic equation,
we use the \emph{Mellin transform} 
\cite{Flajolet95mellintransforms}, which, for a function 
$f:\R \to \R$ is given by
\[
    \Me{f}(s) = \int_{0}^\infty z^{s-1} f(z) \dee{z}.
\]
Using the Mellin transform identities and defining $T(s) = p^{-s} + q^{-s}$, 
we end up with an expression 
for the Mellin transform $\Me{G_k}(s)$ of $\Po{G}_k(z)$ of the form
\[
    \Me{G_k}(s) = \Gamma(s+1)A_k(s)( p^{-s} + q^{-s})^k
    =\Gamma(s+1)A_k(s) T(s)^k,
\]
where $A_k(s)$ (see (\ref{A_kDefinition}) below) is an infinite series arising from the contributions
coming from the function $\Po{W}_{k,G}(z)$:
\begin{align}
	A_k(s) = \sum_{j=0}^k T(s)^{-j} \sum_{m=j}^\infty T(-m)(\mu_{m,j} - \mu_{m,j-1})\frac{\Gamma(m+s)}{\Gamma(m+1)\Gamma(s+1)},
    \label{A_kInitialDef}
\end{align}
where we define $\mu_{m,-1} = 0$ for all $m$.
Note that it involves 
$\mu_{m,j} - \mu_{m,j-1}$ for various $m$ and $j$ 
(see  \cite{magnerPhD2015,magnerknesslszpa2014}).  
Locating and characterizing the singularities
of $\Me{G_k}(s)$ then becomes important.  We find
that, for any $k$, $A_k(s)$ is entire, with zeros at 
$s \in \Z \lintersect [-k, -1]$, so that $\Me{G_k}(s)$ is meromorphic,
with possible simple poles at the negative integers less than $-k$.  
The fundamental strip of $\Po{G}_k(z)$ then contains $(-k-1, \infty)$.
It turns out that the main asymptotic contribution comes from an
infinite number of saddle points (see (\ref{SaddlePoints}) below) defined by the kernel $T(s)= p^{-s} + q^{-s}$.	

We then must asymptotically invert the Mellin transform to recover $\Po{G}_k(z)$.  The Mellin inversion
formula for $\Me{G_k}(s)$ is given by
\begin{align}
    \label{MellinInversionFormula}
    \Po{G}_k(z) = \frac{1}{2\pi i} \int_{\rho - i\infty}^{\rho + i\infty} 
z^{-s}\Me{G_k}(s)\dee{s}
    = \frac{1}{2\pi i} \int_{\rho - i\infty}^{\rho + i\infty} 
z^{-s}\Gamma(s+1)A_k(s)T(s)^k\dee{s},
\end{align}
where $\rho$ is any real number inside the fundamental strip 
associated with $\Po{G}_k(z)$.
For $k$ in the range in which the profile grows polynomially (that coincides with
the range of interest in our analysis of $D_n$),
we evaluate this integral via the saddle point method 
\cite{Flajoletsedgewick2009}.  Examining         
$z^{-s}T(s)^k$ and solving the associated saddle point equation
\[
    \frac{\dee{ }}{\dee{s}} [k\log T(s) - s\log z] = 0,
\]
we find an explicit formula (\ref{RhoDefinition}) below for 
$\rho(\alpha)$, the real-valued saddle point of our integrand.  
The multivaluedness
of the complex logarithm then implies that there are \emph{infinitely many} 
regularly spaced saddle points $s_j$, $j \in \Z$, on this vertical
line:
\begin{align}
	\label{SaddlePoints}
	s_j = \rho(\alpha) + i\frac{2\pi j}{\log(p/q)}.
\end{align}
These lead directly to 		
oscillations in the $\Theta(1)$ factor in the final asymptotics
for $\mu_{n,k}$).  The main challenge in completing the saddle point 
analysis is then to elucidate
the behavior of $\Gamma(s+1)A_k(s)$ for $s\to\infty$ along 
vertical lines: it turns out that this function inherits
the exponential decay of $\Gamma(s+1)$ along vertical lines, 
and we prove it by splitting the sum defining $A_k(s)$
into two pieces, which decay exponentially for different reasons 
(the first sum decays as a result of the superexponential
decay of $\mu_{m,j}$ for $m = \Theta(j)$, which is outside 
the main range of interest).  We end up with an asymptotic
expansion for $\Po{G}_k(z)$ as $z\to\infty$ in terms of $A_k(s)$.

Finally, we must analyze the convergence properties of 
$A_k(s)$ as $k\to\infty$.  We find that it converges
uniformly on compact sets to a function $A(s)$ (see (\ref{A_kDefinition})) 
which is, because of the uniformity, entire.  We then
apply Lebesgue's dominated convergence theorem to conclude that we 
can replace $A_k(s)$ with $A(s)$ in the
final asymptotic expansion of $\Po{G}_k(z)$.  
All of this yields the following theorem which is proved in
\cite{magnerPhD2015,ms2015}.

\begin{theorem}[Moments and limiting distribution for $B_{n,k}$ for $k$ in the central region]
\label{CentralRangeTheorem}
    Let $\epsilon > 0$ be independent of $n$ and $k$, and fix
    $\alpha \in \left(\frac{1}{\log(1/q)}+\epsilon, \frac{1}{\log(1/p)}-\epsilon\right)$. 
    Then for $k = k_{\alpha,n} \sim \alpha \log n$:
    \parsec
    {\rm (i)} The expected external profile becomes
    \begin{align}
    \label{ws1}
        \E[B_{n,k}] 
        &= H(\rho(\alpha), \log_{p/q}(p^kn))
        \cdot \frac{n^{\beta(\alpha)}}{\sqrt{2\pi\kappa_*(\rho(\alpha))\alpha\log n}}
        \left( 1 + O(\sqrt{\log n}) \right), 
    \end{align}
    where
    \begin{align}
        \label{RhoDefinition}
        \rho(\alpha) = -\frac{1}{\log(p/q)}\log\left( \frac{\alpha\log(1/q)-1}
        {1-\alpha\log(1/p)} \right), &&     
        \beta(\alpha) = \alpha\log(T(\rho(\alpha))) - \rho(\alpha),  
    \end{align}
    and $\kappa_*(\rho)$ is an explicitly known function of $\rho$.
Furthermore, $H(\rho, x)$ (see Figure~\ref{HFigure}) 
    is a non-zero periodic function with period $1$ in $x$ given by
    \begin{align}
        H(\rho, x) = \sum_{j\in\Z} A(\rho+it_j)\Gamma(\rho+1+it_j)e^{-2j\pi i x}, 
    \end{align}
    where $t_j = 2\pi j/\log(p/q)$, and
    \begin{align}
        \label{A_kDefinition}
        A(s) = \sum_{j=0}^\infty T(s)^{-j} \sum_{n=j}^\infty T(-n)(\mu_{n,j} - \mu_{n,j-1})\frac{\phi_n(s)}{n!},
    \end{align}
    where $\phi_n(s) = \prod_{j=1}^{n-1} (s+j)$ for $n > 1$ and $\phi_n(s) = 1$ 
    for $n \leq 1$.  We recall that $T(s) = p^{-s} + q^{-s}$.  
    Here, $A(s)$ is an entire function which is zero at the negative integers.
    
    \parsec
    {\rm (ii)} The variance of the profile is
    $
        \Var[B_{n,k}] =\Theta(\E[B_{n,k}]).
    $
\parsec
    {\rm (iii)} The limiting distribution  of the normalized profile is Gaussian;
    that is,
    \[
        \frac{B_{n,k} - \mu_{n,k}}{\sqrt{\Var[B_{n,k}]}} \xrightarrow{D} \Normal(0, 1)
    \]
    where $\Normal(0,1)$ is the standard normal distribution.
\end{theorem}

We should point out that the unusual behavior of $D_n$ in R\'enyi's 
problem is a direct
consequence of the oscillatory behavior of the profile, which disappears for
the symmetric case. Furthermore, for the height and fillup level analyses we
need to extend Theorem~\ref{CentralRangeTheorem} beyond its 
original central range for $\alpha$, as discussed
in the next section.

\section{Proof sketches}
\label{ProofSketches}
Now we give sketches of the proofs of Theorems~\ref{HeightFillupTheorem} 
and \ref{DepthTheorem} with more details regarding the proof of
Theorem~\ref{HeightFillupTheorem} in the forthcoming journal version \cite{drmotamagnerszpa2016}.  
In particular, in this conference version, we only sketch derivations for
$H_n$ and for $F_n$ by upper and lower bounding, respectively.
As stated earlier, the proof of Theorem~\ref{CentralRangeTheorem} 
can be found in \cite{magnerPhD2015,ms2015}.

\subsection{Sketch of the proof of Theorem~\ref{HeightFillupTheorem}}
To prove our results for $H_n$ and $F_n$, we extend the analysis of $B_{n,k}$
to the boundaries of the central region 
(i.e., $k\sim \log_{1/p} n$ and $k\sim\log_{1/q} n$).

\med
{\bf Derivation of $H_n$}.
Fixing any $\epsilon > 0$, we write, for the lower bound on the height, 
\[
	k_L = \log_{1/p} n + (1-\epsilon)\psi(n)
\]    
and, for the upper bound,
\[
	k_U = \log_{1/p} n + (1 + \epsilon)\psi(n),
\]    
for a function $\psi(n) = o(\log n)$ which we are
to determine.  In order for the first and second moment methods to work, we require
$  %
    \mu_{n,k_L} \xrightarrow{n\to\infty} \infty %
$
and %
$
    \mu_{n,k_U} \xrightarrow{n\to\infty} 0.  %
$    
(We additionally need that $\Var[B_{n,k_L}] = o(\mu_{n,k_L}^2)$, but this is not too hard to show by induction
using the recurrence for $\Po{V}_k(z)$, the Poisson variance of $B_{n,k}$.)
In order to identify the $\psi(n)$ at which this transition occurs, we define $k = \log_{1/p} n + \psi(n)$, and
the plan is to estimate $\E[B_{n,k}]$ via the integral representation (\ref{MellinInversionFormula}) for its
Poisson transform.  Specifically, 
we consider the inverse Mellin integrand for some $s=\rho \in \Z^{-} + 1/2$ 
to be set later.  This is sufficient for the
upper bound, since, by the exponential decay of the $\Gamma$ function, 
the entire integral is at most of the same 
order of growth as the integrand on the real axis.  
We expand the integrand in (\ref{MellinInversionFormula}), that is,
\begin{align}
    \label{J_kFormula}
    J_k(n, s) 
    := \sum_{j=0}^k n^{-s}T(s)^{k-j} \sum_{m \geq j} T(-m)(\mu_{m,j} - \mu_{m,j-1}) \frac{\Gamma(m+s)}{\Gamma(m+1)},
\end{align}
and apply a simple extension of Theorem~2.2, part (iii) of      
\cite{magnerknesslszpa2014} to approximate $\mu_{m,j}-\mu_{m,j-1}$         
when $j\to\infty$ and is close enough to $m$:  
\begin{lemma}[Precise asymptotics for $\mu_{n,k}$, $k\to\infty$ and $n$ near $k$]
    \label{KnesslEstimate}
    Let $p \geq q$.  For $n\to\infty$ with $1 \leq k < n$ and $\log^2(n-k)=o(k)$,  
    \begin{align}
        \mu_{n,k}
        \sim (n-k)^{3/2 + \frac{\log q}{\log p}} 
\frac{n!}{(n-k)!}p^{k^2/2 + k/2}q^{k} \cdot 
\exp\left( -\frac{\log^2(n-k)}{2\log(1/p)} \right)\Theta(1).
    \end{align}

    Moreover, for $n \to \infty$ and $k < n$, for some constant $C > 0$,
    \begin{align*}
        \mu_{n,k}
        \leq C\frac{n!}{(n-k-1)!}p^{k^2/2 + k/2 + O(log(n-k)^2)} q^{k}. 
    \end{align*} 
\end{lemma}

Now, we continue with the evaluation of (\ref{J_kFormula}).
The $j$th term of (\ref{J_kFormula}) is then of order
$
    p^{\nu_j(n, s)},
$
where we set
\begin{align*}
    \nu_j(n, s) 
    = ~& (j-\psi(n))^2/2 
      + (j-\psi(n))(s + \log_{1/p}(1 + (p/q)^s) + \psi(n) + 1) \\
      &- \log_{1/p} n \log_{1/p} (1 + (p/q)^s) + \psi(n)^2/2 + o(\psi(n)^2).
\end{align*}
The factor $T(s)^{k-j}$ ensures that the bounded $j$ terms are negligible.

Our next goal is to find the $j$ which gives the dominant contribution to the sum
in (\ref{J_kFormula}); that is, the $j$ for which the contributions 
$p^{\nu_j(n,s)}$ dominate.
By elementary calculus, we can find the $j$ term
which minimizes $\nu_j(n,s)$:   
\begin{align*}
    j =  -(s + \log_{1/p}(1 + (p/q)^s) + 1).  
\end{align*}
Then $\nu_j(n, s)$ for this value of $j$ becomes
\begin{align}
	\nu_j(n, s) 
    &= -\frac{(s + \log_{1/p}(1 + (p/q)^s) + \psi(n) + 1)^2}{2}  \nonumber \\
    &~ ~ ~~- \log_{1/p} n \log_{1/p} (1 + (p/q)^s) + \psi(n)^2/2 + o(\psi(n)^2).
    \label{OptimalNu}
\end{align}
We then minimize over all $s$, which requires us to split into the symmetric
and asymmetric cases.

\paragraph{Symmetric case:}
When $p=q=1/2$, we have $\log_{1/p}(1 + (p/q)^s) = \log_2(2) = 1$, so that the
expression for $\nu_j(n, s)$ simplifies, and we get $s = -\psi(n) + O(1)$.
The optimal value for $\nu_j(n, s)$  then becomes
\begin{align}
    \label{JSumMaxValueSymmetric}
    \nu_j(n, s)
    = -\log_{2} n + \psi(n)^2/2 + o(\psi(n)^2).
\end{align}
We have thus succeeded in finding a likely candidate for the range of $j$ terms 
that contribute maximally, as well as an upper bound on their contribution.  This
gives a tight upper bound on $J_k(n, s)$ and, hence, on $\Po{G}_{k}(n)$, of
$\Theta(2^{-\nu_j(n, s)})$.  

Now, to find $\psi(n)$ for which there is a phase transition in this bound from
tending to $\infty$ to tending to $0$, 
we set the exponent in the above expression equal to zero and
solve for $\psi(n)$.  
This gives
\[
    -\log_2 n + \psi(n)^2/2(1 + o(1)) = 0
    \implies
    \psi(n) \sim \sqrt{2\log_{2} n},
\]
as expected.

\paragraph{Asymmetric case:}
On the other hand, when $p > 1/2$, the equation that we need to solve to find
the minimizing value of $s$ for (\ref{OptimalNu}) is a bit more complicated,
owing to the fact that $\log_{1/p}(1 + (p/q)^s)$ now depends on $s$: taking
a derivative with respect to $s$ in (\ref{OptimalNu}) and setting this equal
to $0$, after some algebra, we must solve
\begin{align}
    -\frac{(p/q)^s\log(p/q)}{\log(1/p)}\log_{1/p} n - \psi(n)(1 + O((p/q)^s)) - s(1 + O((p/q)^s))
    + O((p/q)^s)
    = 0
    \label{SolveForS}
\end{align}
for $s$.
Here, we note that we used the approximation
\[
	\log_{1/p}(1 + (p/q)^s)
    = \frac{(p/q)^s}{\log(1/p)} + O((p/q)^2),
\]
which is valid since we are looking for $s \to -\infty$.

To find a solution to (\ref{SolveForS}), we first note that it implies that 
$s < -\psi(n)$ (since the first term involving $\log n$ is negative), 
and, if $\psi(n) > 0$, this implies that
\begin{align}
	-\psi(n) - s = -O(s).
    \label{PsiSInequality}
\end{align}
The plan, then, is to use this to guess a solution $s$ for (\ref{SolveForS}), which
we can then verify.  The equality (\ref{PsiSInequality}) suggests that we replace
$-\psi(n) - s + O((p/q)^s)$ with $-C\cdot s$ in (\ref{SolveForS}), for some constant $C > 0$. 
Then the equation becomes
\begin{align*}
	-Cs  -\frac{(p/q)^s\log(p/q)}{\log(1/p)}\log_{1/p} n = 0.
\end{align*}
After some trivial rearrangement and multiplication of both sides by $\log(p/q)$, we get
\begin{align*}
	-s\log(p/q) \cdot e^{-s\log(p/q)} = \Theta(\log n).
\end{align*}
Setting $W = -s\log(p/q)$ brings us to an expression of the form that defines the Lambert
$W$ function \cite{AbramowitzStegun} (i.e., a function $W(z)$ satisfying $W(z)e^{W(z)} = z$).

Using the asymptotics of the $W$ function for large $z$ \cite{AbramowitzStegun},
we thus find that 
\begin{align*}
	s = -\log_{p/q}\log n + O(\log\log\log n).
\end{align*}
Note that $s \to -\infty$, as required.  This may be plugged into (\ref{OptimalNu}) to 
see that it is indeed a solution to the equation.

Now, to find the correct choice of $\psi(n)$ for which there is a phase transition, we
plug this choice of $s$ into (\ref{OptimalNu}), set it equal to $0$, and solve for
$\psi(n)$.  This gives
\begin{align}
	\psi(n) 
    = -\frac{s}{2}
    = \frac{1}{2}\log_{p/q}\log n + O(\log\log\log n),
\end{align}
as desired.

Note that replacing $\psi(n)$ in (\ref{OptimalNu}) with $(1+\epsilon)\psi(n)$ yields
a maximum contribution to the inverse Mellin integral of
\begin{align}
	J_{k_U}(n, s)
    = O(p^{\frac{\epsilon}{2}(\log_{p/q}\log n)^2 + o((\log\log n)^2)}) \to 0.
\end{align}
When we replace $\psi(n)$ with $(1-\epsilon)\psi(n)$, we get
\begin{align}
	J_{k_L}(n, s)
    = O(p^{-\frac{\epsilon}{2}  (\log_{p/q}\log n)^2 + o((\log\log n)^2)}),
\end{align}
so that the upper bound tends to infinity (in \cite{drmotamagnerszpa2016}, we prove a
matching lower bound).

The above analysis gives asymptotic estimates for $\Po{G}_k(n)$.  We then apply analytic 
depoissonization \cite{szpa2001Book} to get
\begin{align*}
	\mu_{n,k}
    = \Po{G}_k(n) - \frac{n}{2}\Po{G}''_k(n) + O(n^{\epsilon-1}),
\end{align*}
(where the second term can be handled in the same way as the first).
This gives the claimed result.


\med
{\bf Derivation of $F_n$}.
We now set $k = \log_{1/q} n + \psi(n)$ and
\begin{align}
    \label{k_Lk_UFillup}
    k_L = \log_{1/q} n + (1+\epsilon)\psi(n), &&
    k_U = \log_{1/q} n + (1 - \epsilon)\psi(n).
\end{align}
Here, $\psi(n) = o(\log n)$ is to be determined so as to satisfy
$\mu_{n,k_L} \to 0$ and $\mu_{n,k_U} \to \infty$.
We use a technique similar to that used in the height proof to determine
$\psi(n)$, except now the $\Gamma$ function asymptotics play a role, since we will
choose $\rho \in \R$ tending to $\infty$.  Our first task is to upper bound (as tightly as possible), for
each $j$, the magnitude of the $j$th term of (\ref{J_kFormula}).  First, we upper bound
\begin{align}
    T(-m)(\mu_{m,j} - \mu_{m,j-1})
    \leq 2p^m \mu_{m,j}
    \leq 2p^m m,
    \label{am1}
\end{align}
using the boundary conditions on $\mu_{m,j}$.  Next, we apply Stirling's formula to get
\begin{align}
    \frac{\Gamma(m+\rho)}{\Gamma(m+1)} 
    &\sim \sqrt{1 + \rho/m} \left( \frac{m+\rho}{e}\right)^{m+\rho} \left(\frac{m+1}{e} \right)^{-(m+1)} \\
    &= e^{(m+\rho)\log(m+\rho) - (m+\rho) + m+1 - (m+1)\log(m+1) + O(\log\rho)} \\
    &= \exp( (m+\rho)\log(m+\rho) - (m+1)\log(m+1) + O(\rho)) \\
    &= \exp( m\log(m(1 + \rho/m)) + \rho\log(\rho(1 + m/\rho)) - m\log m - \log m + O(\rho)) \\
    &= \exp( m\log(1 + \rho/m) + \rho\log(\rho) + \rho\log(1 + m/\rho) - \log m + O(\rho)). \label{am2}
\end{align}
Multiplying (\ref{am1}) and (\ref{am2}), then optimizing over all $m \geq j$, we find that the 
maximum term of the $m$ sum occurs at $m = \rho p/q$ and has a value of
\begin{align}
    \label{am3}
    \exp(\rho\log\rho + O(\rho)).
\end{align}
Now, observe that when $\log m \gg \log\rho$, the contribution of the $m$th term is 
$p^{m + o(m)} = e^{-\Theta(m)}$.  Thus, setting $j' = \rho^{\log\rho}$ (note that $\log j' = (\log\rho)^2 \gg \log\rho$),
we split the $m$ sum into two parts:
\[
    \sum_{m \geq j} 2p^m m \frac{\Gamma(m+\rho)}{\Gamma(m+1)}
    = \sum_{m=j}^{j'} 2p^m m \frac{\Gamma(m+\rho)}{\Gamma(m+1)}
    + \sum_{m = j'+1}^{\infty} 2p^m m \frac{\Gamma(m+\rho)}{\Gamma(m+1)}.
\]
The terms of the initial part can be upper bounded by (\ref{am3}), 
while those of the final part are upper bounded by $e^{-\Theta(m)}$ 
(so that the final part is the tail of a geometric series).
This gives an upper bound of
\[
    j' e^{\rho\log\rho + O(\rho)} + e^{-\Theta(j')}
    = e^{(\log\rho)^2 + \rho\log\rho + O(\rho)}
    = e^{\rho\log\rho + O(\rho)},
\]
which holds for any $j$.

Multiplying this by $n^{-\rho}T(\rho)^{k-j} = q^{\rho\cdot(j-\psi(n)) + (j - \psi(n) - \log_{1/q} n) \log_{1/q} (1 + (q/p)^\rho)}$ 
gives
\begin{align}
    \label{am4}
    q^{\rho(j-\psi(n)) + (j-\psi(n) - \log_{1/q} n)\log_{1/q}(1 + (q/p)^{\rho}) - \rho\log_{1/q}\rho + O(\rho)}.
\end{align}
Maximizing over the $j$ terms, we find that the largest contribution
comes from $j=0$.  Then, just as in the height upper bound, the behavior with respect
to $\rho$ depends on whether or not $p = q$, because $\log_{1/q}(1 + (q/p)^{\rho}) = 1$ when $p=q$ and is
dependent on $\rho$ otherwise.  Taking this into account and minimizing over $\rho$ gives that the maximum 
contribution to the $j$ sum is minimized by setting $\rho=2^{-\psi(n)-\frac{1}{\log 2}}$ when $p=q$ and
$\rho \sim \log_{p/q}\log n$ otherwise. 
Plugging these choices for $\rho$ into the exponent of (\ref{am4}), setting it equal to $0$, and solving
for $\psi(n)$ gives $\psi(n) = -\log_2\log n + O(1)$ when $p=q$ and $\psi(n) \sim -\log_{1/q}\log\log n$
when $p > q$.  The evaluation of the inverse Mellin integral with $k=k_L$ as defined in (\ref{k_Lk_UFillup})
and the integration contour given by $\Re(s) = \rho$ proceeds along lines similar to the height proof,
and this yields the desired result.

We remark that the lower bound for $F_n$ may also be derived by relating it to the analogous quantity in 
regular tries: by definition of the fillup level, there are no unary paths above the fillup level in a standard
trie.  Thus, when converting the corresponding PATRICIA trie, no path compression occurs above this level, which
implies that $F_n$ for PATRICIA is lower bounded by that of tries (and the typical value for tries is the same
as in our theorem for PATRICIA).  We include the lower bound for $F_n$ via the bounding of the inverse Mellin
integral because it is similar in flavor to the corresponding proof of the upper bound (for which no short proof
seems to exist).

The upper bound for $F_n$ can similarly be handled by an exact evaluation of the inverse Mellin transform.


\subsection{Proof of Theorem~\ref{DepthTheorem}}
Using Theorem~\ref{CentralRangeTheorem}, we can prove Theorem~\ref{DepthTheorem}.  

\med{\bf Convergence in probability: }
For the typical value of $D_n$, we show that
\begin{align}
    \Pr[D_n < (1-\epsilon)\frac{1}{h(p)}\log n] \xrightarrow{n\to\infty} 0, &&
    \Pr[D_n > (1+\epsilon)\frac{1}{h(p)}\log n] \xrightarrow{n\to\infty} 0.
\end{align}
For the lower bound, we have
\[
    \Pr[D_n < (1-\epsilon)\frac{1}{h(p)}\log n]
    = \sum_{k=0}^{\floor{(1-\epsilon)\frac{1}{h(p)}\log n}} \Pr[D_n = k]
    = \sum_{k=0}^{\floor{(1-\epsilon)\frac{1}{h(p)}\log n}} \frac{\mu_{n,k}}{n}.
\]
We know from Theorem~\ref{CentralRangeTheorem} and 
the analysis of $F_n$ that, in the range of this sum,
$\mu_{n,k} = O(n^{1-\epsilon})$.  Plugging this in, we get
\[
    \Pr[D_n < (1-\epsilon)\frac{1}{h(p)} \log n]
    = \sum_{k=0}^{\floor{(1-\epsilon)\frac{1}{h(p)}\log n}} O(n^{-\epsilon})
    = O(n^{-\epsilon}\log n) = o(1).
\]

The proof for the upper bound is very similar, except that we appeal to the
analysis of $H_n$ instead of $F_n$.

\med{\bf No almost sure convergence: }
To show that $D_n/\log n$ does not converge almost surely, we show that
\begin{align}
    \label{D_nNoAlmostSure}
    \liminf_{n\to\infty} D_n/\log n = 1/\log(1/q), &&
    \limsup_{n\to\infty} D_n/\log n = 1/\log(1/p).
\end{align}
For this, we first show that, almost surely, $F_n/\log n \xrightarrow{n\to\infty} 1/\log(1/q)$
and $H_n/\log n \xrightarrow{n\to\infty} 1/\log(1/p)$.  Knowing this, we consider the following
sequences of events: $A_n$ is the event that $D_n = F_n+1$, and $A'_n$ is the event that $D_n = H_n$.
We note that all elements of the sequences are independent, and $\Pr[A_n], \Pr[A'_n] \geq 1/n$.  This
implies that $\sum_{n=1}^\infty \Pr[A_n] = \sum_{n=1}^\infty \Pr[A'_n] = \infty$, so that the Borel-Cantelli
lemma tells us that both $A_n$ and $A'_n$ occur infinitely often almost surely (moreover, $F_n < D_n \leq H_n$
by definition of the relevant quantities).  This proves (\ref{D_nNoAlmostSure}).

To show the claimed almost sure convergence of $F_n/\log n$ and $H_n/\log n$, we cannot apply the
Borel-Cantelli lemmas directly, because the relevant sums do not converge.  Instead, we apply a trick which was
used in \cite{Pittel85}.  
We observe that both $(F_n)$ and $(H_n)$ are non-decreasing sequences.  Next, we show that, on some appropriately chosen
subsequence, both of these sequences, when divided by $\log n$, converge almost surely to their respective
limits.  Combining this with the observed monotonicity yields the claimed almost sure convergence, and,
hence, the equalities in (\ref{D_nNoAlmostSure}).

We illustrate this idea more precisely for $H_n$.  By our analysis above, we know that
\[
    \Pr[|H_n/\log n - 1/\log(1/p)| > \epsilon] = O(e^{-\Theta(\log\log n)^2}).
\]
Then we fix $t$, and we define $n_{r,t} = 2^{t^2 2^{2r}}$.  
On this subsequence, by the probability
bound just stated, we can apply the Borel-Cantelli lemma to conclude 
that $H_{n_{r,t}}/\log(n_{r,t}) \xrightarrow{r\to\infty} 1/\log(1/p) \cdot (t+1)^2/t^2$ 
almost surely.  Moreover, for every $n$, we can choose $r$ such 
that $n_{r,t} \leq n \leq n_{r,t+1}$.
Then
\[
    H_n/\log n \leq H_{n_{r,t+1}}/\log n_{r,t},
\]
which implies
\[
    \limsup_{n\to\infty} \frac{H_n}{\log n} \leq 
    \limsup_{r\to\infty} \frac{H_{n_{r,t+1}}}{\log n_{r,t+1}} 
\frac{\log n_{r,t+1}}{\log n_{r,t}}
    = \frac{1}{\log(1/p)} \cdot \frac{(t+1)^2}{t^2}.
\]
Taking $t \to \infty$, this becomes $1/\log(1/p)$, as desired.  
The argument for the $\liminf$ is similar,
and this establishes the almost sure convergence of $H_n$.  
The derivation is entirely similar for $F_n$.

\med{\bf Asymptotics for probability mass function of $D_n$: }
The asymptotic formula for $\Pr[D_n = k]$ with $k$ as in the theorem
follows directly from the fact that $\Pr[D_n = k] = \E[B_{n,k}]/n$, plugging in
the expression of Theorem~\ref{CentralRangeTheorem} for $\E[B_{n,k}]$.



\end{document}